\documentclass[11pt]{article}
\usepackage[total={6.5in, 9in}]{geometry}
\usepackage{mathrsfs}
\usepackage{t1enc}
\usepackage[latin1]{inputenc}
\usepackage[english]{babel}
\usepackage{amsmath,amsthm,amssymb,thmtools}
\usepackage{amsfonts}
\usepackage{latexsym}
\usepackage{enumerate}
\usepackage{dsfont}

\usepackage{graphicx}
\usepackage[natural]{xcolor}
\usepackage{tikz}
\usepackage{tikz-3dplot}
\usetikzlibrary{shapes}
\usetikzlibrary{arrows,decorations.pathmorphing,backgrounds,positioning,fit,petri,automata}
\usetikzlibrary {positioning}

\usepackage[
pdfauthor={},
pdftitle={},
pdfstartview=XYZ,
bookmarks=true,
colorlinks=true,
linkcolor=blue,
urlcolor=blue,
citecolor=blue,
bookmarks=true,
linktocpage=true,
hyperindex=true
]{hyperref}

\theoremstyle{plain}
\newtheorem{thm}{Theorem}[section]

\newtheorem{conj}[thm]{Conjecture}
\newtheorem{claim}[thm]{Claim}

\newtheorem{LEM}[thm]{\textbf{Lemma}}

\theoremstyle{plain}

\theoremstyle{plain}

\theoremstyle{plain}

\newcommand{\floor}[1]{\lfloor #1 \rfloor}

\DeclareMathOperator{\df}{def}

\DeclareMathOperator{\sun}{sun}

\begin{document}
\title{Vertex-distinguishing edge coloring of graphs}
\author{Yuping Gao\footnote{Lanzhou University, School of Mathematics and Statistics, Lanzhou 730000, China.}
	\qquad
	Songling Shan\footnote{Auburn University, Department of Mathematics and Statistics, Auburn, AL 36849, USA.
				}\qquad Guanghui Wang\footnote{Shandong University, School of Mathematics, and State Key Laboratory of Cryptography and Digital Economy Security, Jinan 250100, China.}\qquad Yiming Zhou\footnote{Shandong University, School of Mathematics, Jinan 250100, China.}
}

\date{\today}
\maketitle
\begin{abstract} Let $k \ge 1$ be an integer and let $G$ be a nonempty simple graph.
An \emph{edge-$k$-coloring} $\varphi$ of $G$ is an assignment of colors from $\{1,\ldots,k\}$ to the edges of $G$ such that no two adjacent edges receive the same color. For a vertex $v \in V(G)$, we write $\varphi(v)$ for the set of colors assigned to the edges incident with $v$. The coloring $\varphi$ is called \emph{vertex-distinguishing} if $\varphi(u) \ne \varphi(v)$ for every pair of distinct vertices $u,v \in V(G)$.
A vertex-distinguishing edge-$k$-coloring exists if and only if $G$ has at most one isolated vertex and no isolated edge. The least integer $k$ for which such a coloring exists is called the \emph{vertex-distinguishing chromatic index} of $G$, denoted $\chi'_{vd}(G)$.
In 1997, Burris and Schelp conjectured that for every graph $G$ with at most one isolated vertex and no isolated edge,
$ k(G) \;\le\; \chi'_{vd}(G) \;\le\; k(G)+1$,
where $k(G)$ is the natural lower bound required for a vertex-distinguishing coloring in $G$.
In 2004, Balister, Kostochka, Li, and Schelp verified the conjecture for graphs $G$ satisfying
$\Delta(G) \ge \sqrt{2|V(G)|} + 4 $ and
$\delta(G) \ge 5$.
For graphs that do not satisfy these conditions, the best known general upper bound on $\chi'_{vd}(G)$ remains  $|V(G)| + 1$, established in 1999 by Bazgan, Harkat-Benhamdine, Li, and Wo\'zniak.
In this paper, we prove that $\chi'_{vd}(G) \le  \floor{5.5k(G)+6.5}$,
which represents a substantial improvement over the bound $|V(G)| + 1$ whenever $k(G) = o(|V(G)|)$.
We further show that $\chi'_{vd}(G)  \le  k(G) + 3$,
for all $d$-regular graphs $G$ with $d \ge \log_2 |V(G)|\geq 8$.
\medskip

\end{abstract}

\section{Introduction}
  Let  $G$ be a simple graph.
We  denote  by  $V(G)$ and  $E(G)$ the vertex set and edge set of $G$ respectively. For two integers $p$ and $q$, let $[p,q]=\{i\in \mathbb{Z}: p\leq i\leq q\}$.
Given an integer $k\ge 1$,
 an \emph{edge-$k$-coloring} of $G$ is an assignment   of colors from $[1,k]$ to the edges of $G$ such that no two adjacent edges receive the same color.
 The \emph{chromatic index} of  $G$, denoted  by $\chi'(G)$, is the smallest integer $k$ for which $G$ admits an edge-$k$-coloring. 	In the 1960s, Gupta~\cite{Gupta-67}  and  Vizing~\cite{Vizing-2-classes}  independently proved the fundamental result that  $\Delta(G) \le \chi'(G) \le \Delta(G)+1$.  In the 1990s,
 Burris and Schelp~\cite{BS1997} and, independently, Aigner, Triesch, and Tuza~\cite{ATT1992} introduced a variant of edge coloring known as vertex-distinguishing edge coloring, which we introduce below.

 For an \emph{edge-$k$-coloring} $\varphi$ of $G$ and a vertex $v\in V(G)$,
 let $\varphi(v)=\{\varphi(e):  \text{$e$ is  incident with $v$ in $G$}\}$ be the set of colors assigned to edges incident with $v$, and we call it the \emph{color-set} of $v$.
 If $\varphi(u) \ne \varphi(v)$ for any two distinct vertices $u,v$ of $G$, then   $\varphi$  is called a \emph{vertex-distinguishing edge-$k$-coloring}, or a \emph{$k$-vdec} for short.   It is clear that $G$ has a $k$-vdec for some integer $k$ if and only if \emph{$G$  is vdec}, that is, $G$ has at most one isolated vertex
 and  has no isolated edge.  The smallest integer $k$
 for which  a vdec graph has a $k$-vdec is called  the \emph{vertex-distinguishing chromatic index} of $G$, and is denoted
by $\chi'_{vd}(G)$.
 For $d\in [\delta(G), \Delta(G)]$, let $n_d$ be the number of
 vertices of degree $d$ in $G$, and let $k(G) =\min\{k\in \mathbb{N}:  {k \choose d} \ge n_d, d\in [\delta(G), \Delta(G)]\}$.
 Then $k(G)$ is a natural lower bound for $\chi'_{vd}(G)$.
 Burris and Schelp~\cite{BS1997} proposed a conjecture which is
 a vertex-distinguishing analogue of the bound on the chromatic index  of a graph  by  Gupta~\cite{Gupta-67}  and  Vizing~\cite{Vizing-2-classes}.

 \begin{conj}[Burris and Schelp~\cite{BS1997}]\label{conj:VDEC}  Let $G$ be a vdec graph.  Then $ k(G)\le \chi'_{vd}(G) \le  k(G)+1$.
\end{conj}

Solving another conjecture of Burris and Schelp proposed in~\cite{BS1997},  in 1999,
Bazgan,   Harkat-Benhamdine,   Li, and  Wo\'zniak showed in~\cite{BHLW1999} that $\chi'_{vd}(G)\leq |V(G)|+1$  for all vdec graphs $G$.  Later in 2001,  they showed in~\cite{BHLW2001} that $\chi'_{vd}(G)\leq \Delta(G)+5$  if $G$ is a graph with $\delta(G)>\frac{|V(G)|}{3}$.  This minimum degree condition was improved to a degree-sum condition by
Liu and Liu in 2010~\cite{LL2010}.
Balister,   Kostochka,   Li, and  Schelp in 2004 made significant progress on Conjecture~\ref{conj:VDEC}
by establishing it for  graphs $G$ with $\Delta(G)\geq \sqrt{2|V(G)|}+4$ and $\delta(G)\geq 5$~\cite{BKLS2004}.
Recently, in~\cite{GSW2025},  the first three authors determined $\chi'_{vd}(G)$ exactly for large dense regular graphs.  In this paper,   we  provide a new upper bound on $ \chi'_{vd}(G) $
as follows.  The new bound is a substantial improvement over    $|V(G)|+1$
 if $k(G) =o(|V(G)|)$.

\begin{thm}\label{thm:main}
    Let $G$ be a vdec graph.  Then $ \chi'_{vd}(G) \le  \floor{5.5k(G)+6.5}$.
\end{thm}

For regular graphs $G$ with relatively large degree, we can get a better upper bound on $ \chi'_{vd}(G)$
as stated below.

 \begin{thm}\label{thm:main2}
    Let $G$ be a $d$-regular   graph   with $d\ge \log_2 |V(G)|\geq 8$.
    Then $ \chi'_{vd}(G) \le  k(G)+3$.
\end{thm}

The remainder of the paper is organized as follows.
In the next section, we prove Theorems~\ref{thm:main} and~\ref{thm:main2}.
The proof of Theorem~\ref{thm:main}  is based on
 two  new preliminary results, which we prove
 in the last section.

\section{Proof of Theorems~\ref{thm:main} and~\ref{thm:main2}}

Let $G$ be a graph and $S\subseteq V(G)$.  Then $G[S]$
is the  subgraph of $G$ induced by  $S $, and $G-S=G[V(G)\setminus S]$.
For a vertex $x \in V(G)$, we write $G-x$ for $G-\{x\}$.
For two disjoint subsets $A,B\subseteq V(G)$, let $E_G(A,B)=\{uv\in E(G): u\in A, v\in B\}$ be the set of edges of $G$ with one endvertex in $A$ and the other in $B$.   We write $E_G(u,B)$
for $E_G(\{u\}, B)$.
For $F\subseteq E(G)$, we write $G[F]$ for the subgraph of $G$ induced by the edge set $F$, and let $G-F$ be obtained from $G$ by removing all the edges of $F$. For an edge $e\in E(G)$,
we also write $G-e$ for $G-\{e\}$.

Let $k\ge 1$ be an integer, and $\varphi$ be an edge-$k$-coloring of $G$.
For $X\subseteq V(G)$, we  say
that $\varphi$ is \emph{vertex-distinguishing \emph{(}vd\emph{)} on $X$} if $\varphi(u) \ne \varphi(v)$
for any two distinct vertices $u,v\in X$; and that $\varphi$ is \emph{semi-vertex-distinguishing \emph{(}semi-vd\emph{)}  on $X$} if  for any $S\subseteq [1,k]$, there are at most two distinct $u,v\in X$
for which $\varphi(u) = \varphi(v)=S$.

For a subset $S\subseteq [1,k]$,
let  $n_\varphi^S  $ be the number of vertices of $G$
 whose color-set is $S$ under $\varphi$.
We say $\varphi$ is an \emph{optimal edge-$k$-coloring} if it
 minimizes the quantity $\sum_{S\subseteq [1,k]}(n_\varphi^S)^2$. Such optimal colorings clearly exist whenever $k\geq \chi'(G)$. In~\cite{BKLS2004}, Balister, Kostochka, Li, and Schelp established the following result.

\begin{thm}\label{thm:optimal}{\rm (Balister, Kostochka, Li, and Schelp~\cite{BKLS2004})} In any optimal edge-$k$-coloring $\varphi$ of $G$, $|n_\varphi^S-n_\varphi^{S'}|\leq 2$ for all subsets $S,S'\subseteq [1,k]$ with $|S|=|S'|$.
\end{thm}

We will also need the two lemmas below in order to prove
Theorem~\ref{thm:main}.

\begin{restatable}{lem}{mainthm1}\label{thm:color-path}
 Let $F$ be  a linear forest  whose components are
   paths of lengths $2$, $3$, or $4$,  and  let $B$ be  a set of
disjoint $2$-element subsets of $V(F)$.  Suppose that $F$ has an edge-$k$-coloring $\varphi$
such that for any $v\in V(F)$, there is a set  $B_\varphi(v)$ of at most $2{k \choose d_F(v)}-2$  vertices of $F$
that all have the same degree as $v$ in $F$ such that  $\varphi(u) \ne \varphi(v)$ for any $u\in B_\varphi(v)$.
Then there is an edge-$\floor{3.5k+1}$-coloring $\psi$ of $F$
such that  $\psi(u) \ne \psi(v)$ for each pair $\{u,v\}\in B$,   and
$\psi(u) \ne \psi(v)$ for each $u\in B_\varphi(v)$.
\end{restatable}

\begin{restatable}{lem}{mainthm}\label{thm:linear-forest}
  Every graph  $G$ has a linear forest $F$ satisfying the following properties:
    \begin{enumerate}[{\rm (1)}]
        \item  Each component of $F$ is a path of length $2$, $3$, or $4$;
        \item $G-V(F)$ consists of isolated vertices or  isolated edges, and each vertex of $V(G)\setminus V(F)$ has degree at most $(\Delta(G)+1)/2$ in $G$;
        \item For each $v\in V(G)\setminus V(F)$ and $u\in N_G(v)\cap V(F)$,  we have $d_F(u)=2$.
    \end{enumerate}
\end{restatable}

 \proof[Proof of Theorem~{\rm\ref{thm:main}}] Let $G$ be a vdec graph, and let  $k =k(G)+1$ be an integer.
 By Lemma~\ref{thm:linear-forest},  $G$ has a linear forest $F$
that satisfies the three properties as listed in Lemma~\ref{thm:linear-forest}.
  Let $X=V(G)\setminus V(F)$.

 As $k(G) \ge \Delta(G)$, it follows that $k\ge \chi'(G)$ by Vizing's Theorem on  the chromatic index of a graph.
 Thus $G$ has an optimal edge-$k$-coloring $\varphi$.
 We first show that  $\varphi$ is semi-vd on $V(G)$.
Suppose that there exists a color-set $S \subseteq [1,k]$ with $|S|=d$ and $n_{\varphi}^S \geq 3$.
Since ${k \choose d}\geq {k(G) \choose d} \geq n_d$, there must exist another set $S' \subseteq [1,k]$ with $|S'|=d$ and $n_{\varphi}^{S'}=0$. Therefore  $n_{\varphi}^S -n_{\varphi}^{S'} \ge 3$,  contradicting Theorem~\ref{thm:optimal}.
Thus $\varphi$ is semi-vd on $V(G)$.

 If  $X =\emptyset$ or $\varphi$ is vd on $X$, then we proceed directly to the next step.
 Thus we assume for now that $X \ne \emptyset$ and $\varphi$ is not vd on $X$.
 Let
 $$B_X=\{\{u,v\}\subseteq X: u\neq v,\varphi(u)=\varphi(v)\}$$
 be the set of pairs of distinct vertices of $X$ that have the same color-set under $\varphi$.
 As $G$ has at most one isolated vertex and has no isolated edge, for each $\{u,v\} \in B_X$,
 we know that $d_G(u)=d_G(v)\ge 1$ and $d_G(u)=d_G(v)\ge 2$ if $uv\in E(G)$.

 We will recolor some edges incident with exactly one vertex from each element of $B_X$ by some new colors
 to obtain an edge coloring that is vd on $X$ and still semi-vd on $V(G)$.
For each pair $\{u,v\}\in B_X$, we select an edge $e_{u}$ as follows.  If $E_G(u, V(G)\setminus X) \ne \emptyset$, we let $e_u\in E_G(u, V(G)\setminus X)$.
If $E_G(u, V(G)\setminus X) = \emptyset$ but $E_G(v, V(G)\setminus X) \neq \emptyset$,
we let $e_u\in E_G(v, V(G)\setminus X)$. If $E_G(u, V(G)\setminus X) =E_G(v, V(G)\setminus X) = \emptyset$,
then as each component of $G[X]$ is either a single vertex
or an edge, we know that $d_G(u)=d_G(v)=1$ and each of $u$  and $v$ is contained in an edge-component of  $G[X]$.
Since  $G$ has no isolated edge, we know that $v$ is not the neighbor of $u$ in $G$.
In this case, we let $e_u$ be the unique edge incident with $u$ in $G$. Furthermore,
if there is $\{u,v\} \in B_X$ such that  $e_u\in E_G(u, V(G)\setminus X)$ is the selected edge and $u$ is not the endvertex of any other selected edge $e_w$,  $v$ is contained in an
edge component $vz$ of $G[X]$, and $e_z=zv$ is also a selected edge, we then delete $e_u$
from our selection but use $e_z$ to replace $e_u$ (so $e_z$ is a selected edge for two elements of $B_X$).
Let $E_X=\{e_u: \{u,v\} \in B_X\}$ be the collection of all such edges we selected,
and let $H$ be the spanning subgraph of $G$ with $E(H) =E_X$. Then $H$ is a bipartite graph with maximum degree at most $\max\{2, \Delta(G)-2\}$
by Property (3) of $F$.   By our choice of the edges in $E_X$, for any $\{u,v\}\in B_X$,    we have   $\{d_H(u), d_H(v)\} \in \{\{1,0\}, \{2,0\}, \{2,1\}\}$.
For each edge $e\in E(H)$, we recolor $e$ by $\varphi(e)+k$.  Denote by $\psi_0$ this edge coloring of $H$.
Note that $\psi_0(u) =\emptyset$ if $u$ is  an isolated vertex in $H$, and that we introduced at most $k$ new colors $k+1, \ldots, 2k$.

Denote by $\varphi_1$ the new    edge-$(2k)$-coloring of $G$ that is obtained from $\varphi$ by recoloring the edges of $H$ under $\psi_0$.   We next show that
for any two vertices $u,v\in V(G)$,
$\varphi(u) \ne \varphi(v)$ implies that $\varphi_1(u)  \ne \varphi_1(v)$.
Suppose otherwise that $\varphi_1(u) =\varphi_1(v)$.
Then we must have $\psi_0(u)=\psi_0(v)$ and $\varphi_1(u)\setminus \psi_0(u) =\varphi_1(v)\setminus \psi_0(v)$, i.e., $\varphi(u)\setminus \psi_0(u) =\varphi(v)\setminus \psi_0(v)$.  This implies that $\varphi(u)=\varphi(v)$, a contradiction.
Thus $\varphi_1$ is still semi-vd on $V(G)$.
By our construction of $H$, for any $\{u,v\}\in B_X$, we have $\{d_H(u), d_H(v)\} \in \{\{1,0\}, \{2,0\}, \{2,1\}\}$,
and so $\varphi_1(u) \ne \varphi_1(v)$ for any $\{u,v\}\in B_X$. For any two vertices $u,v\in X$ with $\{u,v\}\not\in B_{X}$, we also have  $\varphi_1(u) \ne \varphi_1(v)$ by the above statement as $\varphi(u)\neq \varphi(v)$. Hence $\varphi_1$
is vd on $X$.  Let
$$B=\{\{u,v\}: u,v\in V(G), u\ne v, \varphi_1(u)=\varphi_1(v)\}.$$
 Let $\varphi_1^F$ be the restriction of $\varphi_1$ to $F$.  As the coloring of $F$ under $\varphi_1$ is the same as
 that under $\varphi$, the colors on the edges of $F$
 are from $\{1,\ldots, k\}$. Thus $\varphi_1^F$
 is an edge-$k$-coloring of $F$. For each
 $v\in V(F)$, let
 $$B_{\varphi_1}(v) =\{u\in V(F): d_G(u)=d_G(v), d_F(u)=d_F(v), \varphi_1(u) \ne \varphi_1(v),
 \varphi_1(u)\setminus \varphi_1^F(u) =\varphi_1(v)\setminus \varphi_1^F(v)\}.$$
Since $\varphi_1$ is semi-vd on $V(G)$ and any color-set is used on at most two vertices,
then  $|B_{\varphi_1}(v)| \le 2({k \choose d_F(v)}-1)$ for each $v\in V(F)$.

 Applying Lemma~\ref{thm:color-path}, we have an edge coloring $\psi_1$  that recolors the edges of $F$
 using at most another  $\floor{3.5k+1}$ new colors  $2k+1, \ldots,  2k+\floor{3.5k+1}$ such that
$\psi_1(u) \ne \psi_1(v)$ for each pair $\{u,v\}\in B$ with $u,v\in V(F)$,   and
$\psi_1(u) \ne \psi_1(v)$   for any $v\in V(F)$ and any $u\in B_{\varphi_1}(v)$.
Denote by $\varphi_2$ the new    edge-$\floor{5.5k+1}$-coloring.
We  check below  that $\varphi_2(u) \ne \varphi_2(v)$ for distinct $u,v\in V(G)$, and so $\varphi_2$
is a desired $\floor{5.5k(G)+6.5}$-vdec of $G$.

Consider first that  $u,v\in X$. Then we have $\varphi_2(u) \ne \varphi_2(v)$
as $\psi_1(u) = \psi_1(v) =\emptyset$ and $\varphi_1(u) \ne \varphi_1(v)$  by $\varphi_1$
is vd on $X$.
Consider then that  $u \in X$ and $v \in V(G)\setminus X$.  Then  $\psi_1(u) \ne \psi_1(v)$
and so $\varphi_2(u)  \ne \varphi_2(v)$.
Consider  next that $u,v\in V(G)\setminus X$ but $\{u,v\} \not\in B$. Thus $\varphi_1(u) \ne \varphi_1(v)$.
If $\varphi_1(u)\setminus \varphi_1^F(u) \ne \varphi_1(v)\setminus \varphi_1^F(v)$  or  $\psi_1(u) \ne \psi_1(v)$, then we  have $\varphi_2(u)  \ne \varphi_2(v)$.
Thus we assume that $\varphi_1(u)\setminus \varphi_1^F(u)= \varphi_1(v)\setminus \varphi_1^F(v)$
and $\psi_1(u) = \psi_1(v)$. This implies that $d_G(u)=d_G(v), d_F(u)=d_F(v)$ and so $u\in B_{\varphi_1}(v)$.
However, by Lemma~\ref{thm:color-path}, we should
 have $\psi_1(u) \ne \psi_1(v)$, a contradiction.
 Lastly,  we let $u,v\in V(G)\setminus X$ and $\{u,v\} \in B$.
Then $\psi_1(u) \ne \psi_1(v)$ by Lemma~\ref{thm:color-path}, and so $\varphi_2(u)  \ne \varphi_2(v)$ again.
Therefore $\varphi_2$ is a $\floor{5.5k+1}$-vdec of $G$ and so $\chi'_{vd}(G) \le \floor{5.5k(G)+6.5}$.
 \qed

Let $P_1,\ldots,P_k$ be a family of vertex-disjoint paths.
The collection $\mathcal{P}=\{P_1,\ldots,P_k\}$ is called a \emph{long path system} if $|V(P_i)|\geq 3$ for every $i \in [1,k]$.
The following technical lemma, established by Bazgan, Harkat-Benhamdine, Li, and Wo\'{z}niak~\cite{BHLW1999}, will be
used to prove Theorem~\ref{thm:main2}.

\begin{LEM}\label{lem:semi-VDE}{\rm (Bazgan, Harkat-Benhamdine, Li, and Wo\'{z}niak~\cite{BHLW1999})}
Let $\mathcal{P}=\{P_1,\ldots,P_k\}$ be a long path system and $B$ be a set of
disjoint $2$-element subsets of   $V(\bigcup_{i=1}^k P_i)$.
Then there exists an edge-$3$-coloring  $\varphi$ of $\bigcup_{i=1}^k P_i$ such that $\varphi(x) \ne \varphi(x')$ for each pair $\{x,x'\}\in B$.
\end{LEM}

\proof[Proof of Theorem~{\rm\ref{thm:main2}}]
Let $G$ be a $d$-regular graph on $n$ vertices  with $d \ge \log_2 n\geq 8$.
By the result of Balister,   Kostochka,   Li, and  Schelp in 2004~\cite{BKLS2004},
we may assume that   $d < \sqrt{2n}+4$ and so $G$ is not a complete graph.
As ${2d \choose 0}+{2d \choose 1}+\ldots +{2d \choose 2d}=2^{2d}=4^{d}$, so the largest term ${2d \choose d}\geq \frac{4^{d}}{2d+1} >2^{d}\geq n$ when $d \ge \log_2 n$, it follows that $k(G) \le 2d$.
 By Lemma~\ref{thm:linear-forest},  $G$ has a linear forest $F$
that  covers all the vertices of $G$. Let $H=G-E(F)$. Then $H$ contains at most $\frac{2}{3}n$ vertices of degree $d-1$ and at most $\frac{3}{5}n$ vertices of degree $d-2$. Since $d\ge 8$,
we know that $\delta(H) = d-2 \ge 6$.

 As $k(G) \ge d+2 \ge  \Delta(H)+3$, it follows that $k(G)\ge \chi'(H)$ by Vizing's Theorem on  the chromatic index of a graph.
 Thus $H$ has an optimal edge-$k(G)$-coloring $\varphi$.
 We first show that  $\varphi$ is semi-vd on $V(H)$.
Suppose that there exists a color-set $S \subseteq [1,k(G)]$ with $|S|=s$ and $n_{\varphi}^S \geq 3$.
Since $d\geq 8$ and $k(G) \le 2d$, ${k(G) \choose d} >n$ implies that
\[{k(G) \choose d-1}=\frac{d}{k(G)-d+1}
{k(G) \choose d}\geq \frac{d}{2d-d+1}n
 \geq \frac{2}{3}n\] and   \[{k(G) \choose d-2}=\frac{d(d-1)}{(k(G)-d+1)(k(G)-d+2)}
{k(G) \choose d}\geq \frac{d(d-1)}{(d+1)(d+2)}n
 \geq \frac{3}{5}n.\]
  Thus there   exists another set $S' \subseteq [1,k(G)]$ with $|S'|=s$ and $n_{\varphi}^{S'}=0$. Therefore  $n_{\varphi}^S -n_{\varphi}^{S'} \ge 3$,  contradicting Theorem~\ref{thm:optimal}.
Thus, $\varphi$ is semi-vd on $V(H)$. Let $$B=\{\{u,v\}: u,v\in V(H), u\ne v, \varphi(u)=\varphi(v)\}.$$

 By Lemma~\ref{lem:semi-VDE}, $F$ has   an edge-3-coloring   $\psi$
such that  $\psi(u) \ne \psi(v)$ for each pair $\{u,v\}\in B$ using three new colors $k(G)+1, k(G)+2, k(G)+3$.
Let $\varphi_1$ be the edge-$(k(G)+3)$-coloring  of $G$ obtained by combining $\varphi$ and $\psi$.
We claim  below that $\varphi_1$ is a $(k(G)+3)$-vdec of $G$. It is clear that $\varphi_1$ is an edge-$(k(G)+3)$-coloring of
$G$. It suffices to check that $\varphi_1(u) \ne \varphi_1(v)$ for distinct $u,v\in V(G)$.

 Consider  first  that $u,v\in V(G)$ but $\{u,v\} \not\in B$. Then we have $\varphi(u) \ne \varphi(v)$ and so we  have $\varphi_1(u)  \ne \varphi_1(v)$.
Consider then that  $u,v\in V(G)$ and $\{u,v\} \in B$.
Then $\psi(u) \ne \psi(v)$ by Lemma~\ref{lem:semi-VDE}, and so $\varphi_1(u)  \ne \varphi_1(v)$ again.
Therefore, $\varphi_1$ is a $(k(G)+3)$-vdec of $G$.
\qed

\section{Proofs of Lemmas~\ref{thm:color-path} and~\ref{thm:linear-forest}}

We now prove Lemmas~\ref{thm:color-path} and~\ref{thm:linear-forest}.

\begin{restatable*}{lem}{mainthm1}\label{thm:color-path}
 Let $F$ be  a linear forest  whose components are
   paths of lengths $2$, $3$, or $4$,  and  let $B$ be  a set of
disjoint $2$-element subsets of $V(F)$.  Suppose that $F$ has an edge-$k$-coloring $\varphi$
such that for any $v\in V(F)$, there is a set  $B_\varphi(v)$ of at most $2{k \choose d_F(v)}-2$  vertices of $F$
that all have the same degree as $v$ in $F$ such that  $\varphi(u) \ne \varphi(v)$ for any $u\in B_\varphi(v)$.
Then there is an edge-$\floor{3.5k+1}$-coloring $\psi$ of $F$
such that  $\psi(u) \ne \psi(v)$ for each pair $\{u,v\}\in B$,   and
$\psi(u) \ne \psi(v)$ for each $u\in B_\varphi(v)$.
\end{restatable*}

\proof   For each $v\in V(F)$, let $B_\varphi^*(v)= B_\varphi(v) \cup \{u\}$
if  there exists $u\in V(F)$ such that $\{u,v\} \in B$, and let $B_\varphi^*(v)= B_\varphi(v)$ otherwise.
We will recolor the edges of $F$ using at most  $\floor{3.5k+1}$
colors to obtain a desired coloring $\psi$.  Let $Q_1, \ldots, Q_s$
be all the components of $F$,  where $s\ge 1$  is an integer. We will color the edges
of each $Q_i$ in the order from $Q_1$ to $Q_s$, and we define
the \emph{current color-set} of a vertex of $F$ as an empty set if the component of $F$
that contains this vertex has not been colored.

Let $Q_0$ be the null graph.
For each $i\ge 1$, for $v\in V(\bigcup_{j=1}^{i-1}Q_j)$, let
$B^{i-1}_\varphi(v) = B_\varphi(v)\cap V(\bigcup_{j=1}^{i-1}Q_j)$ and
$B^{i-1}= \{\{u,v\} \in B: u,v \in V(\bigcup_{j=1}^{i-1}Q_j)\}$.
Suppose that we have edge colored $Q_1, \ldots, Q_{i-1}$ for some $i\ge 1$ using at most $\floor{3.5k+1}$ colors, and let
$\psi$ be this partial edge coloring that  satisfies the two conditions with respect to  the sets  $B^{i-1}_\varphi(v)$
and $B^{i-1}$.

Let $Q_i=v_0v_1v_2\ldots v_t$ for  $t\in \{2,3,4\}$. We now extend $\psi$
by coloring
the edges of $Q_i$. All the edges of $Q_i$ will be assigned different colors,
and so we only need to ensure that no vertex $v$ of $Q_i$
has its color-set the same as  that for any vertex from $B^*_\varphi(v) \cap V(\bigcup_{j=1}^{i-1}Q_j)$.
Since $|B^*_\varphi(v_0) \cap V(\bigcup_{j=1}^{i-1}Q_j)| \le 2{k\choose 1}-2+1=2k-1$,  and
we have $\floor{3.5k+1}$ different colors, there is a set $S_0$
of at least $1.5k+2$ colors such that for any $\alpha\in S_0$,
we have $\{\alpha\} \ne \psi(u)$ for any $u\in B^*_\varphi(v_0) \cap V(\bigcup_{j=1}^{i-1}Q_j)$.

Since $|B^*_\varphi(v_1) \cap V(\bigcup_{j=1}^{i-1}Q_j)| \le 2{k \choose 2}-1$, there are at most $1.5k+1$
colors  $\alpha$ from $[1, \floor{3.5k+1}]$  for which at least $k-1$  distinct colors  $\beta \in S_0$ satisfy
 $\{\alpha,\beta\} =\psi(u)$ for some $u\in B^*_\varphi(v_1) \cap V(\bigcup_{j=1}^{i-1}Q_j)$ (since ${k \choose 2}+\lceil0.5k\rceil(k-1) \ge 2{k \choose 2}-1$).
Thus,  there is a set $S_1$ of at least $\floor{3.5k+1}-\floor{1.5k+1}=2k$ colors from $[1,\floor{3.5k+1}]$
such that for any $\alpha \in S_1$,
 there exist at least $1.5k+2-(k-1) \ge 4$ distinct colors $\beta \in S_0$  for which $\{\alpha,\beta\} \ne \psi(u)$ for any  $u\in B^*_\varphi(v_1) \cap V(\bigcup_{j=1}^{i-1}Q_j)$ (note that $k\ge 2$).
 Similarly,  for each $i\in [1,t-1]$,  there is a set $S_i$ of at least $2k$ colors
such that for any $\alpha \in S_i$,
 there exist at least four  distinct colors   $\beta \in S_{i-1}$  for which $\{\alpha,\beta\} \ne \psi(u)$ for any  $u\in B^*_\varphi(v_i) \cap V(\bigcup_{j=1}^{i-1}Q_j)$.

 Since $|B^*_\varphi(v_t) \cap V(\bigcup_{j=1}^{i-1}Q_j)| \le 2k-1$ and $|S_{t-1}| \ge 2k$,
 there is  $\alpha_{t-1}\in S_{t-1}$  such that $\{\alpha_{t-1}\} \ne \psi(u)$ for any $u\in B^*_\varphi(v_t) \cap V(\bigcup_{j=1}^{i-1}Q_j)$. For the color $\alpha_{t-1} \in S_{t-1}$, by our
 choice of $S_{t-1}$, there are at least four distinct colors $\beta \in S_{t-2}$
 such that  $\{\alpha_{t-1}, \beta \} \ne \psi(u)$ for any $u\in B^*_\varphi(v_{t-1}) \cap V(\bigcup_{j=1}^{i-1}Q_j)$.    We choose $\alpha_{t-2} \in S_{t-2}$ with such property such that $\alpha_{t-2}  \ne \alpha_{t-1}$.
 Similarly, for this choice of $\alpha_{t-2}$, if $t\ge 3$, we can choose $\alpha_{t-3} \in S_{t-3}$
 such that $\{\alpha_{t-2}, \alpha_{t-3}\} \ne \psi(u)$ for any $u\in B^*_\varphi(v_{t-2}) \cap V(\bigcup_{j=1}^{i-1}Q_j)$ and that $\alpha_{t-3} \not\in \{\alpha_{t-1}, \alpha_{t-2}\}$. If $t=4$, for  the  choice of $\alpha_{1}$,  we can choose $\alpha_{0} \in S_{0}$
 such that $\{\alpha_{0}, \alpha_{1}\} \ne \psi(u)$ for any $u\in B^*_\varphi(v_{1}) \cap V(\bigcup_{j=1}^{i-1}Q_j)$
 and that $\alpha_0\not\in \{\alpha_1, \alpha_2, \alpha_3\}$.
 Now we extend $\psi$ by coloring $v_0v_1, \ldots, v_{t-1}v_t$ respectively using colors $\alpha_0, \ldots, \alpha_{t-1}$.

 Repeat the same procedure for $Q_{i+1}$ if $i\le s-1$, we can find an edge-$\floor{3.5k+1}$-coloring $\psi$ of $F$
 such that $\psi(u) \ne \psi(v)$ for each pair $\{u,v\}\in B$,   and
$\psi(u) \ne \psi(v)$ for each $u\in B_\varphi(v)$.
\qed

We define some terminology and provide some preliminary results for the proof of Lemma~\ref{thm:linear-forest}.
Let $\mathcal{F}$ be a set of connected graphs.  An \emph{$\mathcal{F}$-packing} of a graph $G$
is a  subgraph $H$ of $G$ such that each component of $H$ is an element of $\mathcal{F}$.
If, in addition, $H$ is spanning, then $H$ is called an $\mathcal{F}$-\emph{factor} of $G$. A vertex $v\in V(G)$ is said to be \emph{covered} by $H$ if $v\in V(H)$.
The existence of $\mathcal{F}$-factors when $\mathcal{F}$ consists of long paths is of particular interest,
and there are established  sufficient and necessary conditions for its existence.   To state one
of such results for our proof, we define below some terminology.

A graph $R$ is called \emph{factor-critical} if for every vertex $x\in V(R)$, the graph $R-x$ has a 1-factor. A graph $H$ is called a \emph{sun} if either
$H = K_1$, $H = K_2$ or $H$ is
obtained from a factor-critical graph $R$ by adding, for each vertex $v\in V(R)$,
a new vertex $w:=w(v)$ together with a new edge $vw$  (see Figure~\ref{figure1}).  We call $R$ the \emph{core} of $H$, and
call $w$ the \emph{pendant} of $v$.
It follows that any sun of order at least three has order at least six and is not bipartite. Denote by
  $\sun(G)$ the number of  components of $G$ that are suns and by $P_n$ a path on $n$ vertices. In~\cite{K2003}, Kaneko established a criterion for a graph to contain a $\{P_3, P_4, P_5\}$-factor. A shorter proof of this theorem, together with a formula for the size of a maximum $\{P_3, P_4, P_5\}$-factor, was later given in~ \cite{KKK2004}.

\begin{figure}[!htb]
	\begin{center}
\begin{tikzpicture}

{\tikzstyle{every node}=[draw,circle,fill=black,minimum size=4pt,
                            inner sep=0pt]
 \draw (0,0) node (v1)  {};

 \draw (3.5,0) node (u1)  {}
        -- ++(180:1.5cm) node (u2)  {}
        ;
   \draw (6,0) node (x1)  {}
        -- ++(60:1.5cm) node (x2)  {}
        -- ++(300:1.5cm) node (x3)   {}
         --(x1)
         --++(225:1cm)node (x4)   {};

      \draw (x2)--++(90:1cm) node (x5)   {};
       \draw (x3)-- ++(315:1cm) node (x6)   {};
   }

\end{tikzpicture}
\end{center}
\caption{Three examples of suns}
\label{figure1}
\end{figure}
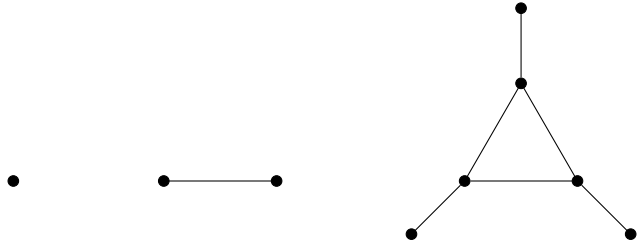

\begin{thm}[Kaneko~\cite{K2003}; Kano, Katona, and Kir\'{a}ly~\cite{KKK2004}]\label{thm:pathfactor}
Let $G$ be a graph. Then $G$ has a $\{P_3, P_4, P_5\}$-factor if and only if
\begin{equation*}
\sun(G-S)\leq 2|S|\ \mbox{for\ all}\ S\subset V (G).
\end{equation*}
\end{thm}

We will also need the following result in the proof of Lemma~\ref{thm:linear-forest}.

\begin{LEM}\label{starfactor}{\rm(Akiyama and Kano~\cite{AK2011}, Theorem 2.10)} Let $G$ be a bipartite graph with partite sets $X$ and $Y$, and $f$ be a function from $X$ to the set of positive integers. If for every subset $S\subseteq X$ we have $|N_G(S)|\geq \sum_{v\in S}f(v)$, then $G$ contains a subgraph $H$ such that $X\subseteq V(H)$, $d_H(v)=f(v)$ for every $v\in X$, and $d_{H}(u)=1$ for every $u\in Y\cap V(H)$.
\end{LEM}

\begin{restatable*}{lem}{mainthm}
 Let $G$ be a graph. Then $G$ has a linear forest $F$ with the following properties:
    \begin{enumerate}[{\rm(1)}]
        \item  Each component of $F$ is a path of length $2$, $3$, or $4$;
        \item $G-V(F)$ consists of isolated vertices or isolated edges, and each vertex of $V(G)\setminus V(F)$ has degree at most $(\Delta(G)+1)/2$ in $G$;
        \item For each $v\in V(G)\setminus V(F)$ and $u\in N_G(v)\cap V(F)$,  we have $d_F(u)=2$.
    \end{enumerate}
\end{restatable*}

 \proof   It suffices to prove the statement for each component of $G$, and so we assume that $G$ is connected.
 Also, the claim holds trivially if $|V(G)| \le 2$. Thus, we assume that $|V(G)| \ge 3$.
 Let  $\df(G) =\max\{\sun(G-S)-2|S|: S\subseteq V(G)\}$ and $\Delta=\Delta(G)$.

 If $G$ has a $\{P_3, P_4, P_5\}$-factor, then the statement holds trivially. Thus we assume that $G$ contains no $\{P_3, P_4, P_5\}$-factor. By Theorem~\ref{thm:pathfactor}, there exists some $S\subseteq V(G)$ such that
$\sun(G-S)> 2|S|$.
The following claim indicates that we may assume $S\ne \emptyset$.

\begin{claim}\label{claim:sun-component}
    Let $D$ be any sun-component of $G-S$ of order at least three.  Then
    \begin{enumerate}[{\rm(1)}]
        \item  $D$ has a $\{P_4,P_5\}$-packing  $F_1$ that only  uncovers a vertex  $y$ that has degree one in $D$,
        and  the neighbor of $y$ in $D$ has degree $2$ in $F_1$.
        \item For any $w\in V(D)$ that is a vertex of the core of $D$, $D$ has a $\{P_4\}$-packing  $F_2$ that only uncovers $w$ and its pendant.
    \end{enumerate}
\end{claim}

\proof Let $R$ be  the core of $D$. For (1), let $x \in V(R)$; for (2), let $x=w$.
Then $R-x$ has a 1-factor $M$ as  $R$
is  factor-critical. Let $t=(|V(R)|-1)/2$ and $M=\{a_1b_1,a_2b_2,\ldots,a_tb_t\}$.
 By the definition of a sun, each $a_i$ has a pendant  $c_i$    and each $b_i$ has a pendant  $d_i$ in $D$. Moreover, $x$ has a neighbor, say $a_j$ in $D$ from $\{a_i,b_i: i\in [1,t]\}$, as well as a pendant $z\in V(D)$. Let $y=c_j$. Then  $F_1$ consisting of the paths $$
c_1a_1b_1d_1, \quad \ldots, \quad c_{j-1}a_{j-1}b_{j-1}d_{j-1}, \quad zxa_jb_jd_j, \quad c_{j+1}a_{j+1}b_{j+1}d_{j+1}, \quad\ldots,\quad c_{t}a_{t}b_{t}d_{t}
$$
is a $\{P_4,P_5\}$-packing  of $D$ that uncovers the only vertex $y$ of $D$. By the construction of $F_1$,
the neighbor $a_j$ of $y$  in $D$ has degree 2 in $F_1$.

The graph  $F_2$ consisting of the paths $$
c_1a_1b_1d_1, \quad \ldots, \quad c_{j-1}a_{j-1}b_{j-1}d_{j-1}, \quad  c_ja_jb_jd_j, \quad c_{j+1}a_{j+1}b_{j+1}d_{j+1}, \quad\ldots,\quad c_{t}a_{t}b_{t}d_{t}
$$
is a $\{P_4\}$-packing  of $D$ that uncovers the vertex $x$  and its pendant.
\qed

By the claim above,  we let  $ \emptyset \ne S \subseteq V(G)$  such that $$\sun(G-S) -2|S| =\df(G).$$

\begin{claim} \label{claim2.5} Let $D$ be a component of $G-S$ that  is not a sun. Then $D$ has a $\{P_3,P_4,P_5\}$-factor.
\end{claim}

\proof  Let $T\subseteq V(D)$. Then
\begin{eqnarray*}
   \df(G) & \ge & \sun(G-(S\cup T))-2|S|-2|T|   \\
   &=& \sun(G-S)+\sun(D-T)-2|S|-2|T| \\
   &=& \df(G)+\sun(D-T)-2|T|,
\end{eqnarray*}
implying  $\sun(D-T)\leq 2|T|$. Thus Theorem~\ref{thm:pathfactor} implies that $D$ has a $\{P_3,P_4,P_5\}$-factor.
\qed

Let $D_1,D_2,\ldots,D_{\ell_1}$ be all the components of $G-S$ and $D_1,D_2,\ldots,D_{\ell}$ be exactly the sun components of $G-S$.  We let $B$ be the bipartite multigraph obtained
from $G$ by contracting each $D_i$ into a single vertex $w_i$ (removing internal edges of $D_i$) for $i\in [1,\ell_1]$
and removing all the edges within $S$. Then $B$ has a bipartition as $S$
and $W:=\{w_1,\ldots,w_{\ell_1}\}$.

\begin{claim}\label{claim2.6} The multigraph  $B$  has a $\{P_3\}$-packing $F_1$   such that $d_{F_1}(s)=2$ for each $s\in S$.
\end{claim}

\begin{proof}
Let $S_1\subseteq S$ and $S_2=S\setminus S_1$. By Lemma~\ref{starfactor}, it suffices to  show that $|N_B(S_1)| \ge  2|S_1|$. Indeed,
\begin{eqnarray*}
    \df(G) &\ge & \sun(G-S_2)-2|S_2|  \\
     &\ge & \sun(G-S)-|N_B(S_1)|-2|S|+2|S_1| \\
      &=& \df(G)-|N_B(S_1)| +2|S_1|,
\end{eqnarray*}
implying $|N_B(S_1)| \ge  2|S_1|$.
\end{proof}

Let  $U=\{u\in W: d_{B}(u)\geq \left\lceil\frac{\Delta}{2}\right\rceil\}$.
We next show that $B$ has a  $\{P_3\}$-packing $F$
covering  $S \cup U$  with vertices of $S$ having degree 2 in $F$.  By Claim~\ref{claim2.6}, we choose a $\{P_3\}$-packing $F_1$ with vertices of $S$ having degree 2
that covers  as many vertices of $U$ as possible. If $U\subseteq V(F_1)$, then we  let $F=F_1$.
Otherwise, we assume that there exists $u\in U$ such that $u\not\in V(F_1)$.  Since $F_1$ is chosen to
cover as many vertices of $U$ as possible, for any $s\in N_B(u)$, the two neighbors of
  $s$ in $F_1$ are from $U$.  We consider all the paths of $B$ starting at $u$
  with edges alternating in non-$F_1$ edges and   $F_1$ edges, and let $U^* \subseteq U$
  and $S^*\subseteq S$ be respectively the set of vertices contained in $U$ and $S$ that are reached by these paths.
  Then we have $u\in U^*$ and $|U^*| =2|S^*|+1$ and $N_B(U^*) \subseteq S^*$.
  However,
  $$ (2|S^*|+1)\cdot \frac{\Delta}{2} = |U^*|\cdot \frac{\Delta}{2}\le |E_B(S^*, U^*)| \le |S^*| \Delta,$$  a contradiction.
  Thus, $B$ has a  $\{P_3\}$-packing $F$
covering  $S \cup U$  with vertices of $S$ having degree 2 in $F$.

We now construct a desired $\{P_3,P_4,P_5\}$-packing of $G$ as follows.
   Let $U=\{u_1,\ldots,u_t\}$ and $D_1,\ldots,D_t$ be the sun components of $G-S$ corresponding to the vertices in $U$, where $t=|U|$.
    We let $F$  be a $\{P_3\}$-packing
that covers $S \cup U$ and with vertices of $S$ having degree 2.
For each $i\in [1,t]$,  let $su_i$ be the edge of $F$  and  let $w_i\in V(D_i)$ such that
$w_is \in E(G)$. Note that $w_i$ is a vertex from the core of $D_i$.  If $|V(D_i)| = 2$, we let  $z_i$ be the neighbor of $w_i$ in $D_i$.
If  $|V(D_i)| \ge 3$, we let  $z_i$ be the pendant  of $w_i$ in $D_i$, then
$D_i$ has a
$\{P_4\}$-packing $F_i$ covering $V(D_i)\setminus \{w_i,z_i\}$
 by Claim~\ref{claim:sun-component}(2).
 We let $F^*$ be obtained from $F$ by replacing $su_i$ with $sw_iz_i$ for each $i\in [1,t]$ when $|V(D_i)| \ne 1$,
 and with $sw_i$ if $|V(D_i)| = 1$.  Then $F^*$ becomes a $\{P_3, P_4, P_5\}$-packing of $G$
 that covers $S$ and one or two vertices from each $D_i$ for $i\in [1,t]$.

For each $D_i$ with $i\in [t+1,\ell]$,   if  $|V(D_i)| \le 2$, then as the vertex  corresponding to $D_i$ in $B$
is not contained in $U$,  it follows that any $v\in V(D_i)$ satisfies $d_G(v) \le (\Delta+1)/2$.
If $|V(D_i)|  \ge  3$, by Claim~\ref{claim:sun-component}(1), we let $F_i$  be a $\{P_4,P_5\}$-packing  of $D_i$  covering all of $V(D_i)$ except for one degree-one vertex $y$.  Furthermore, the neighbor of $y$  in $G$ has degree 2 in $F_i$.
Finally, for each $D_i$ with $i\in [\ell+1,\ell_1]$, by Claim~\ref{claim2.5}, we let $F_i$ be a $\{P_3,P_4,P_5\}$-factor of $D_i$.

 Then the union of $F^*$ and all the $F_i$'s that are defined above gives a $\{P_3,P_4,P_5\}$-packing  $H$ of $G$.
 By the construction, the uncovered vertices by $H$ are either vertices of degree one in $G$ from a  sun-component of $G-S$ and each sun-component of $G-S$ with  one or at least three vertices has at most one such vertex uncovered by $H$, or are two adjacent vertices that form a sun-component of order two of $G-S$.
 Furthermore, by our choice of $U$ in $B$,   for any such uncovered vertex $v$, we have $d_G(v) \le (\Delta+1)/2$.
 Thus $H$ satisfies  Lemma~\ref{thm:linear-forest}(2).

 If there is a vertex $v\in V(G)\setminus V(H)$ and $u\in V(H)$ such that $d_H(u)=1$, then we
 add the vertex $v$ and the edge $uv$ to $H$.  By splitting a $P_6$ into two $P_3$'s in the resulting graph of $H$ if it is needed,  we see that the resulting graph from $H$ still contains a $\{P_3,P_4,P_5\}$-packing of $G$ that satisfies  Lemma~\ref{thm:linear-forest}(1). Thus, by taking $H$ to be a $\{P_3,P_4,P_5\}$-packing  of $G$
 that satisfies Lemma~\ref{thm:linear-forest}(1) with maximal number of vertices, we get a
$\{P_3,P_4,P_5\}$-packing that also satisfies Lemma~\ref{thm:linear-forest}(3).
\qed

\bibliographystyle{abbrv}
\bibliography{vde}

@book{Gupta-67,
	Author = {Gupta, Ram Prakash},
	Mrclass = {Thesis},
	Note = {Thesis (Ph.D.)--Tata Institute of Fundamental Research, Bombay},
	Title = {Studies in the Theory of Graphs},
	Year = {1967}}

@article {Vizing-2-classes,
    AUTHOR = {Vizing, V. G.},
     TITLE = {Critical graphs with given chromatic class},
   JOURNAL = {Diskret. Analiz},
  FJOURNAL = {Akademiya Nauk SSSR. Sibirskoe Otdelenie. Institut Matematiki.
              Diskretny\u{\i} Analiz. Sbornik Trudov},
      YEAR = {1965},
    NUMBER = {5},
     PAGES = {9--17},
   MRCLASS = {05.55},
  MRNUMBER = {200202},
MRREVIEWER = {J. Bos\'{a}k},
}

@article {BS1997,
    AUTHOR = {Burris, A. C. and Schelp, R. H.},
     TITLE = {Vertex-distinguishing proper edge-colorings},
   JOURNAL = {J. Graph Theory},
  FJOURNAL = {Journal of Graph Theory},
    VOLUME = {26},
      YEAR = {1997},
    NUMBER = {2},
     PAGES = {73--82},
      ISSN = {0364-9024},
   MRCLASS = {05C15},
  MRNUMBER = {1469354},
       DOI = {10.1002/(SICI)1097-0118(199710)26:2<73::AID-JGT2>3.0.CO;2-C},
       URL =
              {https://doi.org/10.1002/(SICI)1097-0118(199710)26:2<73::AID-JGT2>3.0.CO;2-C},
}

@incollection {ATT1992,
    AUTHOR = {Aigner, M. and Triesch, E. and Tuza, Z.},
     TITLE = {Irregular assignments and vertex-distinguishing edge-colorings
              of graphs},
 BOOKTITLE = {Combinatorics '90 ({G}aeta, 1990)},
    SERIES = {Ann. Discrete Math.},
    VOLUME = {52},
     PAGES = {1--9},
 PUBLISHER = {North-Holland, Amsterdam},
      YEAR = {1992},
   MRCLASS = {05C15},
  MRNUMBER = {1195794},
       DOI = {10.1016/S0167-5060(08)70896-3},
       URL = {https://doi.org/10.1016/S0167-5060(08)70896-3},
}

@article {BHLW1999,
    AUTHOR = {Bazgan, Cristina and Harkat-Benhamdine, Amel and Li, Hao and
              Wo\'{z}niak, Mariusz},
     TITLE = {On the vertex-distinguishing proper edge-colorings of graphs},
   JOURNAL = {J. Combin. Theory Ser. B},
  FJOURNAL = {Journal of Combinatorial Theory. Series B},
    VOLUME = {75},
      YEAR = {1999},
    NUMBER = {2},
     PAGES = {288--301},
      ISSN = {0095-8956},
   MRCLASS = {05C15},
  MRNUMBER = {1676894},
MRREVIEWER = {Mirko Hor\v{n}\'{a}k},
       DOI = {10.1006/jctb.1998.1884},
       URL = {https://doi.org/10.1006/jctb.1998.1884},
}

@article {BHLW2001,
    AUTHOR = {Bazgan, Cristina and Harkat-Benhamdine, Amel and Li, Hao and
              Wo\'{z}niak, Mariusz},
     TITLE = {A note on the vertex-distinguishing proper coloring of graphs
              with large minimum degree},
      NOTE = {Graph theory (Kazimierz Dolny, 1997)},
   JOURNAL = {Discrete Math.},
  FJOURNAL = {Discrete Mathematics},
    VOLUME = {236},
      YEAR = {2001},
    NUMBER = {1-3},
     PAGES = {37--42},
      ISSN = {0012-365X},
   MRCLASS = {05C15},
  MRNUMBER = {1830596},
       DOI = {10.1016/S0012-365X(00)00428-3},
       URL = {https://doi.org/10.1016/S0012-365X(00)00428-3},
}

@article {LL2010,
    AUTHOR = {Liu, Bin and Liu, Guizhen},
     TITLE = {Vertex-distinguishing edge colorings of graphs with degree sum
              conditions},
   JOURNAL = {Graphs Combin.},
  FJOURNAL = {Graphs and Combinatorics},
    VOLUME = {26},
      YEAR = {2010},
    NUMBER = {6},
     PAGES = {781--791},
      ISSN = {0911-0119},
   MRCLASS = {05C15},
  MRNUMBER = {2729022},
MRREVIEWER = {Futaba Fujie-Okamoto},
       DOI = {10.1007/s00373-010-0949-2},
       URL = {https://doi.org/10.1007/s00373-010-0949-2},
}

@article{BKLS2004,
    AUTHOR = {Balister, P. N. and Kostochka, A. and Li, Hao and Schelp, R.
              H.},
     TITLE = {Balanced edge colorings},
      NOTE = {Dedicated to Adrian Bondy and U. S. R. Murty},
   JOURNAL = {J. Combin. Theory Ser. B},
  FJOURNAL = {Journal of Combinatorial Theory. Series B},
    VOLUME = {90},
      YEAR = {2004},
    NUMBER = {1},
     PAGES = {3--20},
      ISSN = {0095-8956},
   MRCLASS = {05C15},
  MRNUMBER = {2041315},
MRREVIEWER = {Mirko Hor\v{n}\'{a}k},
       DOI = {10.1016/S0095-8956(03)00073-X},
       URL = {https://doi.org/10.1016/S0095-8956(03)00073-X},
}

@article {K2003,
    AUTHOR = {Kaneko, Atsushi},
     TITLE = {A necessary and sufficient condition for the existence of a
              path factor every component of which is a path of length at
              least two},
   JOURNAL = {J. Combin. Theory Ser. B},
  FJOURNAL = {Journal of Combinatorial Theory. Series B},
    VOLUME = {88},
      YEAR = {2003},
    NUMBER = {2},
     PAGES = {195--218},
      ISSN = {0095-8956},
   MRCLASS = {05C70 (05C38)},
  MRNUMBER = {1983352},
MRREVIEWER = {Martin Ba\v{c}a},
       DOI = {10.1016/S0095-8956(03)00027-3},
       URL = {https://doi.org/10.1016/S0095-8956(03)00027-3},
}

@article {KKK2004,
    AUTHOR = {Kano, M. and Katona, G. Y. and Kir\'{a}ly, Z.},
     TITLE = {Packing paths of length at least two},
   JOURNAL = {Discrete Math.},
  FJOURNAL = {Discrete Mathematics},
    VOLUME = {283},
      YEAR = {2004},
    NUMBER = {1-3},
     PAGES = {129--135},
      ISSN = {0012-365X},
   MRCLASS = {05C70},
  MRNUMBER = {2061490},
       DOI = {10.1016/j.disc.2004.01.016},
       URL = {https://doi.org/10.1016/j.disc.2004.01.016},
}

@book {AK2011,
    AUTHOR = {Akiyama, Jin and Kano, Mikio},
     TITLE = {Factors and factorizations of graphs},
    SERIES = {Lecture Notes in Mathematics},
    VOLUME = {2031},
      NOTE = {Proof techniques in factor theory},
 PUBLISHER = {Springer, Heidelberg},
      YEAR = {2011},
     PAGES = {xii+353},
      ISBN = {978-3-642-21918-4},
   MRCLASS = {05-01 (05C70)},
  MRNUMBER = {2816613},
MRREVIEWER = {Sizhong Zhou},
       DOI = {10.1007/978-3-642-21919-1},
       URL = {https://doi.org/10.1007/978-3-642-21919-1},
}

@article {GSW2025,
	AUTHOR = {Gao, Yuping and Shan, Songling and Wang, Guanghui},
	TITLE = {Vertex-distinguishing and sum-distinguishing edge coloring of 
regular graphs},
	JOURNAL = {arXiv:2412.05352v1},
}

\end{document}